\documentclass[12pt]{amsart}
\usepackage{latexsym}

\usepackage[cp1250]{inputenc}
\usepackage[T1]{fontenc} 

\usepackage{graphics}
\usepackage{amsmath}
\usepackage{amssymb}
\usepackage{mathrsfs}
\usepackage{stmaryrd}
\usepackage{eucal}
\usepackage{extarrows}
\usepackage{setspace}

\newtheorem{tw}{Theorem}

\newtheorem{lem}[]{Lemma}
\newtheorem{fa}[]{Fact}
\newtheorem{example}{Example}
\newtheorem{defi}{Definition}

\newfont{\bs}{cmbxsl10 scaled 1200}

\begin{document}
\doublespacing

\title{Asymptotic freeness of Jucys-Murphy elements}

\author{Lech Jankowski}
\email{Lech.Jankowski@math.uni.wroc.pl}
\thanks{Research was supported by the Polish Ministry of Higher Education research grant N N201 364436 for
the years 2009–2012.}


\address{Instytut Matematyczny, Uniwersytet Wrocławski, Pl. Grunwaldzki 2/4, 50-384 Wrocław, Poland}

\begin{abstract}
We give a natural proof of the appearance of free convolution of transition measures in outer product of representations of symmetric groups
by showing that some special Jucys-Murphy elements are asymptotically free.
\end{abstract}
\maketitle

\section*{Introduction}

Almost everything seems to be known about representations of the symmetric groups $S_n$. The answers for basic questions are encoded in the combinatorial structure of  Young diagrams, there is  for example Murnaghan-Nakayama formula for computing characters.
In asymptotic representation theory, as $n$ becomes large, this combinatorial description becomes very complicated and inefficient for concrete calculations so one needs some new methods.

One way of dealing with that is to replace a Young diagram with its \emph{transition measure} (introduced by Kerov [Ker99, Ker03]). 
It is related to the shape of a Young diagram.
The theory is then more analytical and it becomes much easier, for example, to describe the asymptotics of characters.

In 1986 Voiculescu discovered a new type of independence, called \emph{free independence} or just \emph{freeness}. Just like in the case of classical independence of random variables, Voiculescu's freeness leads to a special type of convolution of probability measures on the real line, called \emph{free convolution}.
It turnes out that it can be found in real life situations, for example in random matrix theory.[Spe93]

What is important for this article is that the asymptotic representation theory described in terms of the transition measure is related to Voiculescu's \emph{free probability theory}.

This was first described in [Bia95] in a special case of the left regular representation of $S_n$.
Then in [Bia98] Biane discovered more connections between asymptotic representation theory and free probability, e.g.
he proved that the typical irreducible component of outer product of two irreducible representations of $S_n$ can be 
asymptotically described by the free convolution of their transition measures.
It is an important result as computing outer product of representations is a difficult matter related to Littlewood-Richardson coefficients.
Biane's proof was quite difficult and somehow unnatural, as there was no freely independent random variables.
In this paper we will show that the reason for appearance of free convolution is that some elements of group algebra $\mathbb{C}(S_n)$ called \emph{Jucys-Murphy elements} are asymptotically freely independent.


\section*{Preliminaries}

\subsection{Partitions.}
A {\it partition} of $A=\{1,\dots,m\}$ is any collection \\
$\pi = \{ B_1, \dots, B_k \}$ of pairwise disjoint, nonempty subsets of $A$
such that \\
 $B_1 \cup \dots\cup B_k = A$.
We call the elements of $\pi$ {\it blocks} and denote their number by $|\pi|$.
If every block has two elements we call $\pi$ a {\it pair partition.}
Assume we have a partition $\pi=\{(1,5,6,8), (2,4), (3), (7,9)\}$.
We can draw it in the following way:\\

\includegraphics{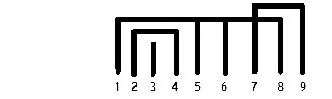}

A partition $\pi$ is said to be \emph{non-crossing} if there exist no \\
$1 \leq k_1 < l_1 < k_2 <l_2 \leq m$ such that $\{k_1,k_2\}, \{l_1,l_2\}$ are contained in two different blocks of $\pi$.
This condition just means that when we draw $\pi$ in the above-described way, the lines do not cross.

A block $B$ of $\pi$ is an {\it inner block} of a block $B'$ if there exist two numbers in $B'$, one smaller than any number in $B$ and one bigger than any number in $B$. We will distinguish between {\it direct} and {\it indirect} inner blocks.
A block $B$ is a {\it direct} inner block of a block $B''$ if there is no block $B'$ with a property that $B$ is an inner block of $B'$ and $B'$ is an inner block of $B''$.
For example, block $\{2,4\}$ is a direct inner blocks of $\{1,5,6,8\}$ because there are no blocks in between whilst block $\{3\}$ is an
\emph{indirect} inner block of $\{1,5,6,8\}$ because there is $\{2,4\}$ in between.

In this article we will use the following notations:
\begin{enumerate}
\item $P(m)$ is the set of all partitions of an $m$-element set.
\item $NC(m)$ is the set of all non-crossing partitions.
\item $NC_{1,2}(m)$ is the set of those non-crossing partitions whose every block
							has cardinality one or two.
\item $NC_{1 < 2}(m)$ is the subset of $NC_{1,2}(m)$ of those partitions
							whose every one-element block is an inner block of some 
							two-element block.
\item $NC_{\geq 2}(m)$ is the set of all non-crossing partitions of an \\
							$m$-element set having blocks of cardinality at least two.
\end{enumerate} 

Let us define a function $F:NC_{1 < 2}(m)\rightarrow NC_{\geq 2}(m)$ by the following procedure:
for every two-element block $B$ of a partition $\pi\in NC_{1 < 2}(m)$ take all its direct inner one-element blocks and 
sum them together with~$B$. This procedure gives us a new non-crossing partition $F(\pi)$ which does not have any one-element blocks.
It is clear that $F$ is a bijection as there is a procedure to get $\pi$ back from $F(\pi)$:
for every block $B$ which has at least three elements take all the elements except the biggest and the smallest one and
put them separately to new blocks.

The following picture shows an example of how the function $F$ works.
\includegraphics{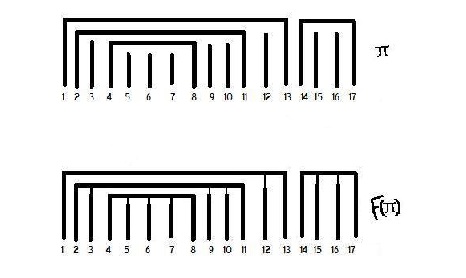}

We can now write
\[
NC_{\geq 2}(m) \xleftrightarrow{F}  NC_{1 < 2}(m) \subset NC_{1,2}(m) \subset NC(m) \subset P(m).
\]

Let $p=(A_1,\dots A_m)$ be a tuple of objects of any kind, some of them may appear many times in the tuple.
We can think of $(A_1,\dots A_m)$ as a \emph{coloring} of a set $\{1,2,\dots,m\}$, namely we think of $A_i$ as the color of the number $i$.
 
We say that a partition 
$\pi$ respects coloring $p$ if no block of $\pi$ contains two numbers with different colors. 
Let $NC^{(p)}(m)$ denote the set of all non-crossing partitions respecting the coloring $p$.
If $\pi \in NC^{(p)}(m)$, then $p$ induces the coloring of blocks of $\pi$.
In this paper the coloring will have only two colors: $X$ and $Y$. 

\subsection{Representation theory}

A {\it group representation} is any homomorphism $\rho$ from a group $G$ to the automorphism group of some vector space $V$
\[
\rho: G\rightarrow \operatorname{Aut}(V)
\]

A {\it character} of a representation $\rho$ is a function $\chi_{\rho}:G \rightarrow \mathbb{C}$ given by
\[
\chi_{\rho}(g)=\operatorname{tr} \rho(g)
\]
where $\operatorname{tr}$ denotes trace of a matrix divided by its dimension, i.e. $$\operatorname{tr}(\rho(g))=\frac{\operatorname{Tr}(\rho(g))}{\operatorname{Tr}(\rho(e))}.$$

A representation is called irreducible if it is not a direct sum of other representations [Ser77].

Two representations $\rho$ and $\rho '$ are equivalent if there are bases in corresponding vector spaces
$V_{\rho}$ and $V_{\rho'}$ such that the matrices $\rho(g)$ and $\rho'(g)$ are equal for all $g\in G$.
It turns out that for a finite group $G$ the maximal number of pairwise inequivalent irreducible representations is finite and
for symmetric groups irreducible representations can be enumerated by objects called \emph{Young diagrams}.
More information about representation theory can be found in [Ser77].

A Young diagram with $n$ boxes can be defined as a descending sequence of nonnegative integers summing up to $n$ or as
a geometric object, namely a finite collection of boxes, or cells, arranged in left-justified rows
with the non-increasing row lengths when we move upwards.
The equivalence is easy to see if we interpret the numbers as lengths of rows. 

\includegraphics{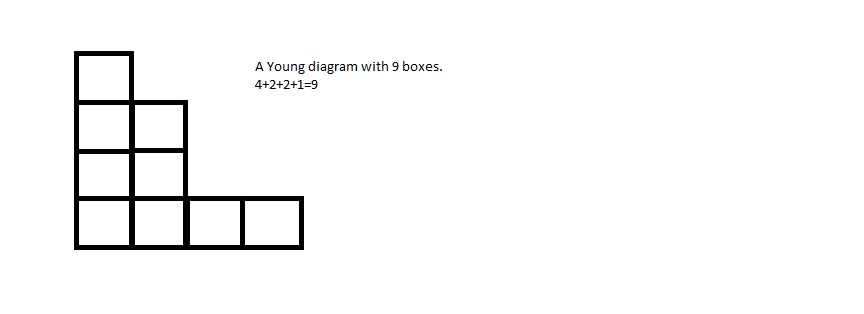}

In asymptotic representation theory the size of the diagram tends to infinity. In the grate majority
of results in this field we consider only so-called \emph{$C$-balanced} diagrams.
Let $C$ be a positive real number.
A Young diagram $(y_1,\dots,y_j)$ with $y_1+\cdots+y_j=n$ is said to be {\it balanced} or, more precisely, {\it C-balanced} if 
$y_i \leq C\sqrt{n}$ for all $1\leq i \leq j$ and $j \leq C\sqrt{n}$.
Sometimes the assumption becomes very precise and we require that the sequence of Young diagrams (rescaled by $\sqrt{n}$ to avoid
growth of the size) tend to some prescribed shape.
In this paper we will not use Young diagrams but we will refer to some dual objects, namely their transition measures.

Asymptotic behaviour of characters of symmetric groups can be described by the following theorem by Biane ([Bia98],[Bia01a])
\begin{fa}
Let $\lambda_n$ be a sequence of $C$-balanced Young diagrams and $\rho_n$ the corresponding representations of $S_n$. Fix a permutation $\sigma\in S_k$ and note
that $\sigma$ can be treated as an element of $S_n$ if we add $n-k$ additional fixpoints. There exists a constant $K$ such that
\[
\big{|}  tr(\rho(\sigma))  \big{|} \leq  K n^{\frac{-|\sigma|}{2}}.
\]
\end{fa}

\subsection{Free probability theory}

We will work in the frame of a non commutative probability space [Voi86] i.e.~a unital algebra $\mathcal{A}$ over $\mathbb{C}$ and a functional $\varphi$, unital in the sense that $\varphi(1)=1$.
We call $a\in \mathcal{A}$ a {\it non-commutative random variable} and $\varphi(a^n)$ its 
{\it moments}. If there are many non-commutative random variables $a_1,\dots,a_k$, their {\it mixed moments} are values of $\varphi$ on words in $a_1,\dots,a_k$.

The parallelism between the classical and the non-commutative probability is easy to see when we let  $\mathcal{A}$ be a commutative algebra of all random variables having all moments and $\varphi$ the classical expectation functional. Most of the classical results like the law of large numbers or central limit theorem can be translated into this language.

In many situations it turnes out that even if $a\in \mathcal{A}$ is not a genuine random variable, the sequence $\varphi(a^n)$ of its moments is a sequence of moments of a probability measure $\mu_a$ on a real line.
We call this measure the distribution of $a$.

One of the fundamental tools in both classical and free probability theory are \emph{classical} and \emph{free cumulants} [NS99].
\begin{defi}
Assume we have a tuple $(A_1,\dots,A_m)$ of non-commutative random variables.
Free cumulants is a family of functions $C_{\pi}(A_1,A_2,\dots,A_m)$ satisfying the following conditions:
\begin{enumerate}
\item $C_{\pi}(A_1,A_2,\dots,A_m)$ factorize according to the block structure~of~$\pi$, i.e.~$C_{\pi}(A_1,A_2,\dots,A_m)=\prod_{b|\pi}C_{b}(A_1,A_2,\dots,A_m)$,
\item $C_{b}(A_1,A_2,\dots,A_m)$ depends only on arguments with indexes from $b$ so if $b=\{i_1,\dots,i_l\}$ then
			$C_{b}(A_1,A_2,\dots,A_m)=C_{l}(A_{i_1},\dots,A_{i_l})$,
\item $C_{l}(A_{i_1},\dots,A_{i_l})$ are multilinear functionals,
\item $C_{\pi}$ satisfy the {\it free moment-cumulant formula}:\\
			$\varphi(A_1,A_2,\dots,A_m)=\sum_{\pi\in NC(m)}C_{\pi}(A_1,A_2,\dots,A_m)$.
\end{enumerate}

\end{defi}

To get the definition of classical cumulants and the {\it classical moment-cumulant formula} we only have to replace the set $NC(m)$ of non-crossing partitions in condition $(4)$ by 
the set $P(m)$ of all partitions. 

If $A=A_1=A_2=\cdots =A_m$ and the distribution of $A$ is some measure $\mu$, then we call $C_m(A,A,\dots,A)$ the $m$-th free cumulant of
$A$ or $m$-th free cumulant of $\mu$ and we write $C_m(A)$ or $C^{\mu}_m$.

\begin{defi}
We say that non-commutative random variables $X_1,\dots,X_k$ are independent in the free way (or simply that they are free) if their mixed free cumulants vanish, i.e. $C_{l}(X_{i_1},\dots,X_{i_l})=0$ whenever at least two $i_j$ are not equal.
\end{defi}

The following definition of freeness is clearly equivalent to the one above and will be more convenient for our purpose:

\begin{defi}
We say that $X_1,\dots,X_k$ are free if for every tuple $(A_1,\dots,A_m)$ of these variables we have
\[
\varphi(A_1,A_2,\dots,A_m)=\sum_{\pi \in NC^{(p)}(m) }C_{\pi}(A_1,A_2,\dots,A_m) ,
\]
where $p$ denotes the coloring of the set $\{1,2,\dots,m\}$ given by $p(i)=A_i$ and
$NC^{(p)}(m)$ is the set of those non-crossing partitions of $\{1,2,\dots,m\}$ which respect
the coloring $p$.
\end{defi}

To get the definition of asymptotic freeness we have to replace equality by a limit:

\begin{defi}  \label{defi:asymptfree}
Let $(\mathcal{A}_n,\varphi_n)$ be a sequence of non-commutative probability spaces and for every $n$ let $X_1^n,\dots,X_k^n \in \mathcal{A}_n$ be a tuple of non-commutative random 
variables. The sequences $(X_1^n),\dots,(X_k^n)$
are said to be asymptotically free if 
\[
\varphi_n(A^n_1,A^n_2,\dots,A^n_m) \rightarrow \sum_{\pi \in NC^{(p)}(m)}C_{\pi}(A_1,A_2,\dots,A_m) 
\]
for every tuple $(A^n_1,\dots,A^n_m)$ of sequences $X_1^n,\dots,X_k^n$. 
\end{defi}

\begin{tw}[Voi86]
If $X,Y$ are free random variables with respective distributions $\mu_X$ and $\mu_Y$, the distribution of their sum depends only on  $\mu_X$ and $\mu_Y$ and is called a free convolution of measures  $\mu_X$ and $\mu_Y$. We denote it by
$\mu_X \boxplus \mu_Y$.
\end{tw}

\begin{tw}[NS99]
If $X_n$ and $Y_n$ are asymptotically free sequences of random variables converging to some free random variables $X$ and $Y$, the sequence of distributions of their sums
$X_n + Y_n$ converges \emph{in distribution} to the distribution of $X+Y$ which is the free convolution  $\mu_X \boxplus \mu_Y$.
Convergence in distribution means that all mixed moments of $X_n$ and $Y_n$ converge to the respective mixed moments of $X$ and $Y$.
\end{tw}

The following fact is a simple consequence of combining Definition 2 with (1) in Definition 1 and will be later used to recognize free cumulants of
free random variables.

\begin{fa}
Let $(A_1,\dots,A_m)$ be a tuple of non-commutative random variables $X$ and $Y$ with respective distributions $\mu_X$ and $\mu_Y$.
As we described before, $(A_1,\dots,A_m)$ defines a coloring of the set $\{1,2,\dots,m\}$ with colors $X$ and $Y$.
If $X$ and $Y$ are free, then for $\pi\in NC^{(p)}(m)$ their free cumulants are given by
$C_{\pi}(A_1,\dots,A_m)=  \prod_{b|\pi_X}C_{|b|}^{\mu_X}\prod_{b|\pi_Y}C_{|b|}^{\mu_Y}$
where $\pi_X$ and $\pi_Y$ are sets of those blocks of $\pi$ which have color $X$ and $Y$ respectively.
\end{fa}

The only non-commutative probability space investigated in this paper will be 
the group algebra $\mathbb{C}[S_{2n+1}]$ of the symmetric group  $S_{2n+1}$
with a functional $\varphi$ defined by 
$$
\varphi(\sigma)= \operatorname{tr}(\rho_1 \varotimes \rho_2 \varotimes \operatorname{id})(\sigma \downarrow_{S_{n}\times S_{n}\times \{e\}}^{S_{2n+1}}).
$$ 
The arrow denotes restriction to a subgroup, so if $\sigma \in S_{2n+1}$ is not an element of
$S_{n}\times S_{n}\times \{e\}$, the value of $\varphi$ on it is by definition equal to zero.
If $\sigma$ is an element of the subgroup, it is a product of two permutations $\sigma_1$ and $\sigma_2$ such that
the support of $\sigma_1$ is contained in $\{1,2,\dots,n\}$ and the support of $\sigma_2$ is contained in $\{n+1,\dots,2n\}$.
The value of $\varphi$ on $\sigma$ is then a product of two characters of irreducible representations $\rho_1$ and $\rho_2$ of a symmetric group $S_n$ evaluated respectively on $\sigma_1$~and~$\sigma_2$.

\subsection{Outer product} 

Let $H$ be a subgroup of a finite group $G$ and let $\rho$ be a representation of $H$ in a space $V$. 
The \emph{induced representation} $U^{\rho}$ of $G$ is defined [FH91] as follows:
\begin{itemize}
\item The set $\{f:G\rightarrow V  : (\forall h\in H)(\forall g\in G)f(hg)=\rho(h)f(g)\}$ is the space of $U^{\rho}$. 
\item $G$ acts by the formula $[U^{\rho}(g_0)f](g)=f(gg_0)$.
\end{itemize}

The {\it outer product} [FH91] of two representations $\rho_1$ and $\rho_2$ of the symmetric group $S_n$ is a representation of $S_{2n}$ induced from a representation 
$\rho_1 \varotimes \rho_2$ of a subgroup $S_n \times S_n$ defined by $(\rho_1 \varotimes \rho_2)(g_1,g_2)=(\rho_1(g_1))\varotimes(\rho_2(g_2))$. 

In this paper what we need to know about the outer product is how to compute its character:

\begin{tw}[Frobenius' Duality Theorem][FH91,NaiS82]
Let $\rho_G$ be a representation of a group $G$ induced from a representation $\rho_H$ of its subgroup~$H$.
Let $\chi_G$ and $\chi_H$ be the corresponding characters and let $A_G=\{gh_0g^{-1}:g\in G\} \in \mathbb{C}[G]$ and $A_H=\{hh_0h^{-1}:h\in G\} \in \mathbb{C}[H]$ be 
conjugacy classes of the same element $h_0$ of $H$ but in different groups. Then
\[
\chi_G (A_G)=\chi_H (A_H) \\
\]
\end{tw}

Thus to compute a character of the outer product of representations it suffices to compute the character of their tensor product, i.e.~the 
product of their characters.

\subsection{Jucys-Murphy elements and transition measure}

In this section we will investigate moments of Jucys-Murphy elements.
A \emph{Jucys-Murphy} element [J66,OV04] is the sum of transpositions $$J_n=(1,n)+(2,n)+ \cdots + (n-1,n)\in\mathbb{C}[S_n].$$ 
The summands will be called \emph{Jucys-Murphy transpositions}.

There are a few equivalent definitions of \emph{Kerov transition measure} but for our purposes the following one is the most convenient.
Given a Young diagram $\lambda$ or, equivalently, an irreducible representation of $S_n$, we define its transition measure $\mu_{\lambda}$ as the unique probability measure on the real line with a property that
$$\displaystyle{ \int_{\mathbb{R}} x^{k} d\mu_{\lambda} =\operatorname{tr}\left(\rho_{\lambda}\left(J^{k}_{n+1}\downarrow_{S_n}^{S_{n+1}}\right) \right)  }.$$

In the following we will consider normalized Jucys-Murphy elements $$X_n:=n^{-\frac{1}{2}}\sum_{i=1}^{n}(i,2n+1)$$ and $$Y_n:=n^{-\frac{1}{2}}\sum_{i=n+1}^{2n}(i,2n+1).$$ 
Note that $X_n$ and $Y_n$ are not genuine Jucys-Murphy elements in $\mathbb{C}[S_{2n+1}]$ as the sums are not taken over $\{1,\dots,2n\}$.
The genuine Jucys-Murphy element is their scaled sum $\sqrt{n}(X_n + Y_n)$. Elements $X_n$ and $Y_n$ can be seen as Jucys-Murphy elements
in subalgebras of $\mathbb{C}[S_{2n+1}]$ generated by sugroups $\{g\in S_{2n+1} : g(i)=i,  \text{for } i\in\{n+1,\dots 2n\}\}$ and 
$\{g\in S_{2n+1} : g(i)=i, \text{for  } i\in\{1,\dots n\}\}$ respectively.
We will call $X_n$ and $Y_n$ Jucys-Murphy elements as it makes sense in the context of outer product of representations.



It follows from linearity that the value of $\varphi$ on a product of a tuple of Jucys-Murphy elements $X_n$ and $Y_n$ is a sum 
of values of $\varphi$ on many products of transpositions. Our aim is to
group the summands to obtain a sum over non-crossing partitions.
We will then get the formula from Definition 4.

We will investigate Jucys-Murphy elements, Jucys-Murphy transpositions and
their moments so now we will try understanding these objects as well as we have to.

Assume we have a tuple $A_1,\dots,A_m$ where every $A_i$ is equal either to $X_n$ or to $Y_n$. What we need to compute is 
\[
\varphi(A_1\cdots A_m).
\]

Each $A_i$ is a sum of Jucys-Murphy transpositions, all together normalized by a factor $n^{-\frac{m}{2}}$. For a tuple $(a_1,\dots,a_m)$ of Jucys-Murphy transpositions we will write $(a_1, \dots, a_m) \sim (A_1, \dots ,A_m)$ if each $a_i$ is a summand 
of the corresponding $A_i$. 

We will need this notation to write the moment of a product of a tuple $\varphi(A_1\cdots A_m)$. 

Each tuple $(a_1,\dots ,a_m)$ defines a partition $\pi$ of the set $\{1,\dots,m\}$ by the following rule: two numbers are in the same block if and only if the corresponding transpositions are equal. We will denote it by $\pi \approx (a_1,\dots,a_m)$.
Of course, for every partition $\pi$ there are many tuples $(a_1,\dots,a_m)$ defining $\pi$.

Each tuple $(A_1,\dots,A_m)$ defines a coloring of the set $\{1,\dots,m\}$.
The color of a number $i$ is simply $A_i$, i.e.~either $X_n$ or $Y_n$.
So if $\pi \approx (a_1,\dots,a_m) \sim (A_1,\dots,A_m)$ then the 
partition $\pi$ must respect the coloring defined by $(A_1,\dots,A_m)$.

\begin{lem}
If $(a'_1,\dots, a'_m) \approx \pi \approx (a_1,\dots,a_m)$ and $(a'_1,\dots, a'_m) \sim (A_1,\dots,A_m) \sim (a_1,\dots,a_m)$, then the products $a'_1\cdots a'_m$ and 
 $a_1\cdots a_m$ are permutations conjugate by an element of $S_n\times S_n \times \{e\}.$
\end{lem}
\begin{proof}
Let $b_1,\dots,b_k$ be blocks of $\pi$. Every block $b_j$ corresponds to some Jucys-Murphy transposition $a_i=(\alpha_i,2n+1)$
and some $a'_i=(\alpha'_i,2n+1)$.
From the assumption that $(a'_1,\dots, a'_m) \sim (A_1,\dots,A_m) \sim (a_1,\dots,a_m)$, both $\alpha_i$ and $\alpha'_i$ belong either to $\{1,2,\dots,n\}$ or $\{n+1,n+2,\dots,2n\}$.

Define a permutation $\gamma$ putting $\gamma(\alpha_i)=\alpha'_i$ for every $\alpha_i$. 
For the other numbers we can put arbitrary
values but we need to make sure that $\gamma \in  S_n\times S_n \times \{e\}$.
It is possible because when we put $\gamma(\alpha_i)=\alpha'_i$ we use exactly the same number of elements from $\{1,2,\dots,n\}$ as arguments and as values.
Any permutation defined in that way satisfies
\[
a_1\cdots a_m = \gamma^{-1} a'_1\cdots a'_m \gamma
\]
\end{proof}

\begin{lem} \label{lemat1}
If two permutations $\sigma$ and $\sigma'$ are conjugate by an element of $S_n\times S_n \times \{e\}$ then $\varphi(\sigma)=\varphi(\sigma')$.
\end{lem}
\begin{proof}
If $\sigma$ is not an element of $S_n\times S_n \times \{e\}$, then neither is $\sigma'$ and by definition
$\varphi(\sigma)=0=\varphi(\sigma')$. 

If $\sigma\in S_n\times S_n \times \{e\}$ then we can split it into commuting $\sigma_1$ and $\sigma_2$ such that
$\sigma = \sigma_1\sigma_2$ and 
$supp(\sigma_1)\subset \{1,2,\dots,n\}$ , $supp(\sigma_2)\subset \{n+1,n+2,\dots,2n\}$.

If $\sigma'=\gamma \sigma \gamma^{-1}$ where $\gamma \in S_n\times S_n \times \{e\}$ then we can analogously split 
$\gamma$ into $\gamma_1$ and $\gamma_2$.

One can then see that the conjugation is really taken in subgroups $S_n\times \{e_{S_n}\} \times \{e\}$ and $\{e_{S_n}\}\times S_n \times \{e\}$, namely
$\sigma'_i=\gamma_i\sigma_i\gamma_i^{-1}$.
By definition the value of $\varphi$ is a product of corresponding characters, i.e. traces of matrices:
\[
tr(\sigma'_i)=tr(\gamma_i\sigma_i\gamma_i^{-1})=tr(\sigma_i)
\]
and thus $\varphi(\sigma)=\varphi(\sigma')$
\end{proof}

\begin{lem}
Let $a_1, \dots, a_m$ be Jucys Murphy transpositions. Assume that for some $k$ we have $a_k\neq a_i$ for all $i\neq k$.
Let $\sigma_1=a_1 a_2 \cdots a_{k-1}$ and $\sigma_2=a_{k+1} \cdots a_m$.
Then $|\sigma_1 a_k \sigma_2|=|\sigma_1 \sigma_2|+1$ where the \emph{length} $|\sigma|$ of a permutation $|\sigma|$ is defined as
the minimal number of factors needed to write $\sigma$ as a product of transpositions.
\end{lem}

\begin{proof}
To prove the lemma it suffices to explain the following equalities: 
\[
|\sigma_1 a_k \sigma_2|=|\sigma_2\sigma_1a_k|=|\sigma_2\sigma_1|+1=
|\sigma_1\sigma_2|+1.
\]
The first two permutations are conjugate hence have the same length. 

The second equality comes from the fact that if  $a_k=(j,2n+1)$, the numbers $j$ and $2n+1$ appear in the permutation $\sigma_2\sigma_1$ in two different cycles. Therefore if we multiply $\sigma_2\sigma_1$ by $a_k$, then $a_k$ glues these two cycles together so the length of their product is bigger by one that the 
length of $\sigma_2\sigma_1$. 

The last equality holds because again the two permutations are conjugate.
\end{proof}

The following Lemma is proved in [Bia95].
\begin{lem}
Suppose $a_1, \dots, a_m$ are Jucys-Murphy transpositions. 
Define a partition $\pi$
of the set $\{1,\dots,m\}$ 
joining in blocks those numbers $j,k$ for which $a_j=a_k$ and assume this is a pair partition. 
Then the product $a_1\cdots a_{m}$ is the identity permutation if and only if   
$\pi  \in NC_2(m)$.
\end{lem}
\begin{proof}
As the if part of the proof is obvious we will only prove the only part.
We know thus that  $a_1...a_{m}$ is the identity permutation and that every factor has exactly one copy in the tuple. We prefer to think of the product $a_1...a_{m}$ as of $m$ transpositions acting 'separately' in the appropriate order.
Let us see what happens with the number $2n+1$ under the action of  $a_1...a_{m}$
Let $a_{m}=(\alpha_m,2n+1)=a_k$ 
The transposition $a_{m}$ sends $2n+1$ to the place of $\alpha_m$ and it stays there until $a_k$ sends it back. 
If $k\neq 1$ then $a_{k-1}$ sends $2n+1$ to the place of some other element $\alpha_{k}$.
Knowing that $a_1\cdots a_{m}$ is the  identity permutation we know that the copy of
$a_{k-1}$ can't be between $a_k$ and $a_m$ because it would make it impossible for $2n+1$ to come back from the place of $\alpha_{k}$.
So the copy of $a_{k-1}$ is some $a_j$ with $j<k-1$. We iterate the same argument until we get to $1$.
Let us now see what happens to the number $\alpha_m$. First it goes to the place of $2n+1$, then the transposition $a_{m-1}=(\alpha_{m-1},2n+1)=a_r$ sends it to the place of $\alpha_{m-1}$. Knowing that $a_1...a_{m}$ is the identity permutation we know that $\alpha_m$ has to be back on the place of the number $2n+1$ right before the action of $a_k$ because it is its only way back to the right place.
It follows that $a_{r}$ is somewhere between $a_{k+1}$ and $a_{m-2}$. We iterate the argument until we get to $a_{k+1}$. Applying the same argument inductive finishes the proof.
\end{proof}

\begin{lem} \label{lemat2}
Let $ (a_1,\dots,a_m) $  be a tuple of Jucys-Murphy transpositions, let  $\pi$  be the 
partition described in the previous Lemma 
and let $\sigma= a_1\cdots a_m$ be the product of the tuple.

Then $2|\pi|\geq|\sigma|+m$ if and only if $\pi\in NC_{1,2}(m).$
\end{lem}
\begin{proof}
If we remove every one-element block and all the corresponding transpositions from the tuple $ (a_1,\dots,a_m) $  , we get another partition $\pi '$ of the $m'$-element set and a new permutation $\sigma '$.

First notice that every one-element block has the same contribution to both sides of the inequality. Indeed, it is clear that its contribution to $m$ and $\pi$ is equal to one and since it corresponds to some transposition which is not repeated in the tuple, Lemma 3 tells us that the contribution to $|\sigma|$ is also one.
Thus the new partition $\pi '$ satisfies the inequality if and only if $\pi$ does.

Assume now that the inequality is true.
The maximal possible value of  $2|\pi '|$ is $m'$ because every block consists at least two elements. 
As we already have $m'$ on the right hand side of the inequality, it follows that $m'$ is also the minimal possible value of $2\pi'$ so $\pi'$ is a pair partition and that $|\sigma '|$ has to be equal $0$. 
Partitions of $\{1,2,...,m'\}$ for which the product of corresponding tuples has length $0$ ( is equal to $e$) are by Lemma 4 
exactly $NC_{2}(m')$. 

Conversely if $\pi'\in NC_{2}(m)$, then $\sigma=e$ and $2|\pi '|=m'$ so the inequality is true.  
\end{proof}

We will now show how to read the cycle structure of a product of Jucys-Murphy transpositions corresponding to a given
partition $\pi\in NC_{1,2}(m)$.
Denote by $\alpha_i$ the element exchanged with $2n+1$ by the transposition $a_i$, i.e. $a_i=(\alpha_i, 2n+1)$.

\begin{lem} \label{lemat3}
Suppose that a tuple $ a_1,\dots,a_m $ corresponds to a partition $\pi\in NC_{1,2}(m)$. 
Let $i_1\leq \cdots \leq i_k$ be the elements of those one-element blocks of $\pi$ which are not 
inner blocks.
Then the product $ a_1 \cdots a_m $ has a cycle $(2n+1,\alpha_{i_k},\dots,\alpha_{i_1})$.
If there are no such one-element inner blocks, then $2n+1$ is a fixed point of the product.
The other nontrivial cycles are in one-to-one correspondence with those two-element blocks of $\pi$ which have some direct inner one-element blocks.
If such a two-element block corresponds to some $a_j$ and the direct inner one-element blocks correspond to $a_{j_1},\dots ,a_{j_l}$ with $j_1\leq \cdots \leq j_l$, then the corresponding cycle is 
$(\alpha_j,\alpha_{j_l},\dots,\alpha_{j_1}).$ 
\end{lem}

It is strongly recommended to read the following example and do the included calculations before or instead of reading the proof.

\begin{example}
Suppose we have a tuple $(a_1,\dots,a_9)$ with $a_2=a_9$ and $a_4=a_7$ and no other equalities.
The corresponding partition is on the picture:

\includegraphics{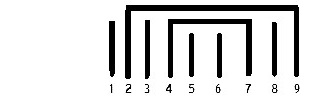}

The product of the tuple:
$$
a_1 \cdots a_9= $$
$$
=(\alpha_1, 2n+1)(\alpha_9, 2n+1)(\alpha_3, 2n+1)(\alpha_7, 2n+1)$$
$$
(\alpha_5, 2n+1)(\alpha_6, 2n+1)(\alpha_7, 2n+1)(\alpha_8, 2n+1)(\alpha_9, 2n+1)
$$
has the following cycle decomposition:
$$(2n+1, \alpha_1)(\alpha_9, \alpha_8,\alpha_3)(\alpha_7,\alpha_6,\alpha_5).$$

\end{example}

\begin{proof}[Proof of Lemma 6]

Let us first see what happens to the element $2n+1$ under the action of the product $(a_1\cdots a_m)$.
The first transposition $a_m$ sends it to the place of the element $\alpha_m$.
If $m$ is in a one-element block of $\pi$, then $2n+1$ will stay at the place of $\alpha_m$.
If $m$ is in a two-element block $\{k,m\}$, then $2n+1$ will return to its original place under the action of
$a_k$. Any $a_l$ with $k< l < m$ will not affect $2n+1$.


So the element $2n+1$ goes to the place of an element $\alpha_i$ with $\{i\}$ being the first non-inner one-element block from 
the right. 
If such an $i$ does not exist, then $2n+1$ stays at its original place.

Note that the same is true for any element that is at the place of the element $2n+1$. 
So if $2n+1$ goes to the place of some $\alpha_i$, then $\alpha_i$ goes to the place of the first
$\alpha_j$ where $\{j\}$ is the first one element block from the right if we restrict the partition $\pi$ to the set
$\{1,2,\dots,i-1\}$.

Using the above-described method we can read the cycle consisting $2n+1$ out of the partition $\pi$.
First we write $2n+1$ and then we choose all non-inner one-element blocks $\{s\}$ of $\pi$ from the right to the left
and write the corresponding $\alpha_s$ after $2n+1$ from the left to the right just like in Example 1.


Let us now see how to construct the other cycles of $a_1\cdots a_m$.
Choose some two-element block $\{p,r\}$ of $\pi$. Note that the elements $\alpha_p, \alpha_{p+1}, \dots, \alpha_r$ 
cannot appear in transpositions $a_1,\dots,a_{p-1},a_{r+1},\dots,a_m$ as $\pi$ is non-crossing.
Consider $\alpha_r$. Under the action of the product $a_1\cdots a_m$ it is first moved by $a_r$ and it goes to the place of the
element $2n+1$. 

Now we can apply the same argument as before with a difference that - as we are inside the block $\{p,r\}$ - instead of looking at 
non-inner one-element blocks we need to look at direct inner one-element blocks of $\{p,r\}$.
So to read the cycle corresponding to the block $\{p,r\}$ out of the partition $\pi$ we need to choose all direct inner one-element
blocks $\{s\}$ of $\{p,r\}$ from the right to the left and write the corresponding $\alpha_s$ after $\alpha_r$ from the left to the right.

\end{proof}

\begin{defi} 
Let $(n)_{k}= n(n-1)\cdots(n-k+1)$ denote the product of descending integers.
\end{defi}

\begin{lem} \label{lemat4}
Suppose we have a tuple $A_1,\dots,A_m$
of Jucys-Murphy elements i.e.~every $A_i$ is equal either to $X_n$ or to $Y_n$ and a partition $\pi$ respecting the corresponding coloring. Then the number of tuples $a_1,\dots,a_m$ such that  $\pi \approx a_1,\dots,a_m \sim A_1,\dots,A_m$ is equal to $(n)_{k}(n)_{l}$ where $k$ is the number of blocks of $\pi$ with color $X$ and $l$ is the number of blocks of $\pi$ with color $Y$.
\end{lem}


The following Lemma is a reformulation of Theorem 1.3 from [Bia98].

\begin{lem} \label{lemat5}
Let balanced Young diagrams $\lambda_1$ and $\lambda_2$ corresponding to $\rho_1$ and $\rho_2$ in the definition of $\varphi$ have, in the limit when
$n$ goes to infinity, some limit shapes
$\Lambda_1$ and $\Lambda_2$.
Let $\sigma$ be a product of some tuple
of Jucys-Murphy transpositions satisfying $\pi \approx (a_1,\dots,a_m) \sim (A_1,\dots,A_m)$ and assume that $\sigma \in S_n \times S_n \times \{e\}$.
Let $\sigma_1, \sigma_2$
be such that $\operatorname{supp}(\sigma_1)\subset\{1,\dots,n\}$, $\operatorname{supp}(\sigma_2)\subset\{n+1,\dots,2n\}$ and $\sigma=\sigma_1 \sigma_2$
where $\operatorname{supp}$ denotes the support of a permutation.

Then $$n^{\frac{|\sigma|}{2}} \varphi(a_1\dots a_m) \rightarrow \prod_{c|\sigma_1}C_{|c|+2}^{\mu_{\Lambda_1}}\prod_{c|\sigma_2}C_{|c|+2}^{\mu_{\Lambda_2}}$$
where $C^{\mu_\Lambda}_k$ denotes the $k$-th free cumulant of $\mu_{\Lambda}$.  
\end{lem}

\section{The result}

Now we are ready to compute the mixed moment of a tuple $(A_1, A_2, \dots, A_m)$:
\[
\varphi(A_1 A_2 \cdots A_m)=
n^{-\frac{m}{2}}\sum_{(a_1, \dots, a_m) \sim (A_1, \dots A_m)}
\varphi(a_1 \cdots a_m).
\]

Every tuple $(a_1,\dots,a_m)$ of transpositions defines a partition $\pi \approx (a_1,\dots,a_m)$ so we can write the above sum grouping the summands according to partitions:

\begin{equation} \label{eq:wz}
\varphi(A_1 A_2 \cdots A_m)=n^{-\frac{m}{2}} \sum_{\pi\in P(m)}
\left(
\sum_{ \pi \approx (a_1,\dots,a_m) \sim (A_1,\dots,A_m)}
\varphi(a_1 \cdots a_m)
\right),
\end{equation}
where the second sum is taken over all tuples $(a_1,\dots ,a_m)$ such that 
$\pi \approx (a_1,\dots,a_m) $ and $(a_1,\dots,a_m) \sim (A_1,\dots,A_m)$.

From Lemmas 1 and 2
we know that the value of $\varphi$ on each tuple
$(a_1,\dots,a_m)$ corresponding to a fixed partition $\pi$ is the same.
We can thus denote this value by $\varphi_{\pi}(A_1, A_2, \dots, A_m)$ and write
$$
\varphi(A_1 A_2 \cdots A_m)= $$
$$
=\sum_{\pi\in P(m)}n^{-\frac{m}{2}}
\#\{(a_1,\dots a_m): \pi \approx (a_1,\dots,a_m) \sim (A_1,\dots,A_m)\}\ \varphi_{\pi}(A_1, A_2, \dots, A_m).
$$
Then from Lemma 7 
we get

\[
\varphi(A_1 A_2 \cdots A_m)= \sum_{\pi\in P(m)}n^{-\frac{m}{2}}\ 
(n)_k\  (n)_l\  \varphi_{\pi}(A_1, A_2, \dots, A_m).
\]

From Fact 1 we know that
for a fixed $\pi$ the absolute value of a summand is bounded :

\[
\big{|}
n^{-\frac{m}{2}}\ 
(n)_k \ (n)_l \ \varphi_{\pi}(A_1, A_2, \dots, A_m)\big{|}
\leq n^{-\frac{m}{2}}n^{|\pi|} C n^{-\frac{|\sigma|}{2}}=
C n^{|\pi|-\frac{|\sigma|}{2}-\frac{m}{2}},
\]
where $\sigma$ is a product of any tuple which satisfies  $\pi \approx (a_1,\dots,a_m) \sim (A_1,\dots,A_m)$ (the value $|\sigma|$ does not depend on the choice of the tuple).
From Lemma 5, the only partitions the contribution of which does not asymptoticaly vanish are partitions from $NC_{1,2}(m)$ and
from Lemma 6 we know how to read permutations out of partitions. The functional $\varphi$ is equal to zero on permutations which do not belong to $S_n \times S_n \times \{e\}$, so in order to get some nonzero
value, every two-element block must have the same color as all its direct inner one-element blocks.
This just means that $F(\pi)$ where $F$ is a function described in section $0.1$ must respect the coloring given by $A_1,\dots,A_m$ so it is better to sum over $F(\pi)$ instead of $\pi$.
We can thus write 
\[
\varphi(A_1 A_2 \cdots A_m)= \sum_{F(\pi)\in NC^{(p)}_{>1}(m)}n^{-\frac{m}{2}}\ 
(n)_{k}\ (n)_{l}\ \varphi_{\pi}(A_1, A_2, \dots, A_m) + o(1).
\]

From Lemma 8 we have:

\[
\varphi(A_1 A_2 \cdots A_m)=
\]
\\
\[
=\sum_{F(\pi)\in NC^{(p)}_{>1}(m)}\underbrace{(n)_{k}(n)_{l}n^{-\frac{m}{2}}n^{-\frac{|\sigma|}{2}}}_{\text{this tends to }1} 
\underbrace{n^{\frac{|\sigma|}{2}}\varphi_{\pi}(A_1, A_2, \dots, A_m)}_{\text{this tends to } \prod_{c|\sigma_1}C_{|c|+2}^{\mu_{\Lambda_1}}\prod_{c|\sigma_2}C_{|c|+2}^{\mu_{\Lambda_2}}}
\rightarrow 
\]
\[
\rightarrow \sum_{F(\pi)\in NC^{(p)}_{>1}(m)} \prod_{c|\sigma_1}C_{|c|+2}^{\mu_{\Lambda_1}}\prod_{c|\sigma_2}C_{|c|+2}^{\mu_{\Lambda_2}}.
\]

From Fact 2 we know that $\prod_{c|\sigma_1}C_{|c|+2}^{\mu_{\Lambda_1}}\prod_{c|\sigma_2}C_{|c|+2}^{\mu_{\Lambda_2}}$ are
free cumulants of a tuple of some free non-commutative random variables with distributions $\mu_{\Lambda_1}$ and $\mu_{\Lambda_2}$.
By Definition 4. this means that $A_1,\dots,A_m$ are asymptotically free.








\newpage
\begin{center}

\textbf{REFERENCES}

\end{center}

\medskip

\noindent [Bia95] Philippe Biane. Permutation model for semicircular systems and quantum random walks.
 								  Pacific J. Math. Volume 171, Number 2 (1995), 373-387.

\medskip

\noindent [Bia98] Philippe Biane. Representations of symmetric groups and free probability. 
 								  Adv. Math., 138(1):126–181, 1998.
\medskip

\noindent [Bia01a] Philippe Biane. Approximate factorization and concentration for characters of
                   symmetric groups. Internat. Math. Res. Notices, (4):179–192, 2001. 

\medskip

\noindent [FH91] William Fulton, Joe Harris 'Representation theory:a first course' Springer 1991

\medskip

\noindent [J66] Jucys, Algimantas Adolfas (1966), "On the Young operators of the symmetric group", Lietuvos Fizikos Rinkinys 6: 163–180

\medskip

\noindent [Ker99] S. Kerov. A differential model for the growth of Young diagrams. In Proceedings
of the St. Petersburg Mathematical Society, Vol. IV, volume 188 of Amer.
Math. Soc. Transl. Ser. 2, pages 111–130, Providence, RI, 1999. Amer. Math.
Soc.

\medskip

\noindent [Ker03] S. V. Kerov. Asymptotic representation theory of the symmetric group and its
applications in analysis, volume 219 of Translations of Mathematical Monographs.
AnsmericanMathematical Society, Providence, RI, 2003. Translated from
the Russian manuscript by N. V. Tsilevich, With a foreword by A. Vershik and
comments by G. Olshanski.

\medskip

\noindent [NaiS82] THEORY OF GROUP REPRESENTATIONS. M. A. Naimark, A. I. Stern [translated by Elizabeth and Edwin. Hewitt]. 
					568 pp. Springer-Verlag, New York, 1982. 

\medskip

\noindent [NS99] A. Nica, R. Speicher, Lectures on the Combinatorics of Free Probability Theory, Paris, 1999.

\medskip

\noindent [OV04] Okounkov, Andrei; Vershik, Anatoly (2004), "A New Approach to the Representation Theory of the Symmetric Groups. 2", 
					Zapiski Seminarod POMI (In Russian) v. 307

\medskip

\noindent [Ser77] Jean-Pierre Serre. Linear Representations of Finite Groups. Number 42 in Graduate
Texts in Mathematics. Springer-Verlag, 1977.

\medskip

\noindent [Spe93] R. Speicher, Free convolution and the random sum of matrices, RIMS 29 (1993),
731–744.

\medskip

\noindent [Voi86] Dan Voiculescu. Addition of certain noncommuting random variables. J. Funct.
Anal., 66(3):323–346, 1986.

\medskip

\end{document}